\documentclass[11pt]{amsart}

 \usepackage{amsfonts,graphics,amsmath,amsthm,amsfonts,amscd,amssymb,amsmath,latexsym,multicol,
 mathrsfs}
\usepackage{epsfig,url}
\usepackage{flafter}
\usepackage{fancyhdr}
\usepackage{hyperref}
\hypersetup{colorlinks=true}

\usepackage{graphicx}
\usepackage{xcolor}
\usepackage{float}
\usepackage{bm}
\usepackage{enumitem}
\usepackage{multirow}
\usepackage{scalerel,stackengine}
\usepackage{stmaryrd}
\usepackage{datetime}
\usepackage[title]{appendix}
\usepackage{url}
\usepackage{tabulary}
\usepackage{booktabs}
\usepackage{tikz}
\usetikzlibrary{positioning}
\usepackage{verbatim}

\usepackage[utf8]{inputenc}
\usepackage[english]{babel}
\usepackage{csquotes}
\usepackage{comment}

\usepackage[
backend=biber,
style=alphabetic,
sorting=anyt,
doi=false,isbn=false,url=false,
maxalphanames=99,maxbibnames=99,backref=true,
]{biblatex}
\DefineBibliographyStrings{english}{
  backrefpage  = {\hspace{-0.35em}},
  backrefpages = {\hspace{-0.35em}},
}

\DeclareFieldFormat{bracketswithperiod}{\mkbibbrackets{#1}}

\renewbibmacro*{pageref}{%
  \iflistundef{pageref}
    {}
    {\printtext[bracketswithperiod]{%
       \ifnumgreater{\value{pageref}}{1}
         {\bibstring{backrefpages}\ppspace}
         {\bibstring{backrefpage}\ppspace}%
       \printlist[pageref][-\value{listtotal}]{pageref}}%
     }}

 \theoremstyle{plain}
 \newtheorem{theorem}{Theorem}[section]
 \newtheorem{lemma}[theorem]{Lemma}
 \newtheorem{corollary}[theorem]{Corollary}
 \newtheorem{proposition}[theorem]{Proposition}

 \theoremstyle{definition}
 \newtheorem{definition}[theorem]{Definition}

 \newtheorem{remark}[theorem]{Remark}
 
 \newtheorem{example}[theorem]{Example}

 \theoremstyle{remark}

 \theoremstyle{plain} 
\newcommand{\thistheoremname}{}
\newtheorem{genericthm}[theorem]{\thistheoremname}

  \newtheorem*{genericthm*}{\thistheoremname}
\newenvironment{namedthm*}[1]
  {\renewcommand{\thistheoremname}{#1}%
   \begin{genericthm*}}
  {\end{genericthm*}}
 
\newcommand{\spec}{\operatorname{Spec}}

\newcommand{\chara}{\operatorname{char}}

\newcommand{\projection}{\operatorname{pr}}

\newcommand{\grass}{\operatorname{Grass}}
\newcommand{\galois}{\operatorname{Gal}}

\newcommand{\fanoind}{\operatorname{Index}}

\newcommand{\mni}{\medskip\noindent}

\newcommand{\mbb}{\mathbb}
\newcommand{\QQ}{\mbb{Q}}
\newcommand{\NN}{\mbb{N}}

\newcommand{\RR}{\mbb{R}}

\newcommand{\PP}{\mbb{P}}

\newcommand{\mc}{\mathcal}

\newcommand{\mf}{\mathfrak}

\newcommand{\ol}{\overline}

\usepackage{graphicx}
\usepackage{pifont}
\usepackage{xcolor}
\usepackage{float}
\usepackage{amssymb}
\usepackage{amsmath}
\usepackage{mathrsfs}
\usepackage{bm}
\usepackage[all,cmtip]{xy}
\usepackage{enumitem}
\usepackage[utf8]{inputenc}
\usepackage[english]{babel}
\usepackage{multirow}
\usepackage{mathtools}
\usepackage{scalerel,stackengine}
\usepackage{stmaryrd}
\usepackage{datetime}
\usepackage[title]{appendix}
\usepackage{url}
\usepackage{tabulary}
\usepackage{booktabs}
\usepackage{tikz}
\usetikzlibrary{positioning}
\usepackage{verbatim}

\title{\large O\MakeLowercase{n} S\MakeLowercase{ubvarieties of} D\MakeLowercase{egenerations of}
 F\MakeLowercase{ano} V\MakeLowercase{arieties}}
\author{\large S\MakeLowercase{antai} \large Q\MakeLowercase{u}}
\begin{document}

\begin{abstract}
The goal of this work is to study geometric properties of 
geometrically irreducible subschemes on degenerations of 
Fano varieties (more generally, of
separably rationally connected varieties).  
It is known that these geometrically irreducible subschemes 
exist when the ground field has characteristic zero
or contains an algebraically closed subfield.  
We show
that the dimension of this geometrically irreducible subscheme
has a lower bound by the Fano index of the generic fibre.
\end{abstract}

\maketitle


\section{Introduction}
Geometric properties of degenerations of Fano varieties are closely related 
to arithmetic properties of the ground field.
A conjecture of James Ax (which we call the \emph{Ax conjecture}) 
predicts that every perfect PAC field is $C_1$ (cf. \cite[Theorem D]{ax68}).
Recall that a field $k$ is called $C_1$ if every homogeneous polynomial $f\in k[x_0, \ldots, x_n]$
of degree $\le n$ has a nontrivial zero in $k^{n+1}$
(\cite[Definition 24.2.1]{fj}), and a field $k$ is called \emph{PAC}
(pseudo-algebraically closed) if each geometrically irreducible $k$-variety has a $k$-rational point
(\cite[page 203]{fj}).  
By showing that the degeneration of Fano varieties must contain a geometrically irreducible subscheme
(over the residue field),
J\'anos Koll\'ar solves the Ax conjecture in characteristic zero completely
(where the idea originates from \cite{DJL83}).
Recall that a \emph{Fano variety} over a perfect field $k$
is a normal projective variety over $k$ so that $-K_X$ is ample 
(note that there is no requirement on the singularities of the variety, cf. \S\ref{RC-SRC}).

\begin{theorem}[\protect{\cite[Theorem 3, page 236]{Kollarcom}}]\label{kollarAx}
Let $k$ be a field of characteristic 0, $C$ a smooth $k$-curve, $Z$ a reduced, 
irreducible, projective $k$-variety,
and $g\colon Z\to C$ a $k$-morphism.  Assume that the generic fibre $F_{\text{gen}}$ is 
\begin{itemize}
	\item [\emph{(1)}] smooth,
	\item [\emph{(2)}] geometrically irreducible, and
	\item [\emph{(3)}] Fano (that is, $-K_{F_\text{gen}}$ is ample).
\end{itemize}
Let $c\in C$ be a closed point with residue field $k (c)$.  Then, the fibre $g^{-1}(c)$
contains a $k(c)$-subscheme which is geometrically irreducible. 
If, in addition, every $k(c)$-irreducible component of $g^{-1}(c)$ is smooth (or normal),
then $g^{-1}(c)$ contains a $k (c)$-irreducible component which is
geometrically irreducible.
\end{theorem}


\begin{corollary}[\protect{\cite[Theorem 1]{Kollarcom}}]\label{Ax-chara-zero}
	Every PAC field of characteristic zero is $C_1$.
\end{corollary}

We include the proof from \cite{Kollarcom} here to illustrate how the arithmetic of the
ground field is concluded from the geometric property of degenerations of Fano varieties
(cf. \cite[\S 5]{Kollarcom}).

\begin{proof}
    Let $k$ be a PAC field of characteristic zero (which is automatically perfect).
	To show that $k$ is $C_1$, we need to 
    establish the existence of a $k$-rational point on the hypersurface,
    which can be singular, $H_f\subset \mathbb{P}_k^n$, defined
    by a homogeneous polynomial $f$ of degree $d$ that is $\le n$.
    By connecting $H_f$ with a smooth hypersurface of the same degree via a pencil
    of hypersurfaces,
    $H_f$ becomes a degeneration of smooth Fano varieties (notice that the degree condition $d\le n$
    gives the Fano property).  However, by Theorem~\ref{kollarAx}, $H_f$
    contains a geometrically irreducible $k$-subscheme, hence $H_f$
    contains a $k$-rational point as $k$ is a PAC field.
\end{proof}


As for the construction in the proof of Corollary~\ref{Ax-chara-zero}, 
many examples of fibrations $X\to Z$ with Fano generic fibre are realised as subfamilies
of the constant family of the projective spaces, that is, there is an embedding 
$X\to Z\times \PP^n$ over $Z$ (see, for example, \cite[Theorem 2, Example 4]{Kollarcom}).
The information about this embedding can be encoded (and generalised) 
into a polarisation by a Cartier divisor under certain conditions.

\begin{definition}\label{polarised-g-fano}
	Let $f\colon X\to Z$ be a contraction (see \S 2) of integral schemes
    where the residue field $k(z)$ is perfect for every (not necessarily closed) point $z\in Z$.
	We say that $f$ is an \emph{$L$-polarised g-Fano fibration} if 
	\begin{itemize}
		\item $L$ is a big$/Z$ and base-point-free$/Z$ Cartier divisor on $X$, and
		\item the geometric generic fibre $\ol{F}_{\text{gen}}$ of $f$ is weak $\QQ$-Fano (that is, $\ol{F}_{\text{gen}}$ has klt singularities and $-K_{\ol{F}_{\text{gen}}}$ is nef and big, see Definition~\ref{weak-Q-Fano}).
	\end{itemize}
\end{definition}

As $L$ is base-point-free$/Z$ and big$/Z$, it is clear that $L$
defines a birational $Z$-morphism $\phi\colon X\to \PP^n_Z$.  
We can measure the positivity of the relative anticanonical divisor
of $X\to Z$ via the polarisation $L$ in the following sense.

\begin{definition}\label{g-fano-index}
	Let $f\colon X\to Z$ be an $L$-polarised g-Fano fibration.
	The \emph{$L$-index of $X\to Z$} (or \emph{the index of $X\to Z$ with respect to $L$}),
	denoted by $i_L(X/Z)$, is the integer $r\ge 0$ such that 
	\begin{equation}
		-K_{F_{\text{gen}}} \equiv r (L|_{F_{\text{gen}}}) \text{ over }K(Z), \label{def-equ}\tag{\dag}
	\end{equation}
	where $F_{\text{gen}}$ is the generic fibre of $f$, $L|_{F_{\text{gen}}}$
    is the pullback of $L$ to $F_{\text{gen}}$, and $K(Z)$ is the function field of $Z$.
	If such an integer $r$ does not exist, we set $i_L(X/Z) = -\infty$.
\end{definition}

Note that \eqref{def-equ} in Definition~\ref{g-fano-index} is defined locally over the generic point
of $Z$; however, the polarisation $L$ is globally defined on the whole scheme $X$ 
of the fibration $X\to Z$.  While, for many examples
regarding degenerations of Fano varieties in practice,
$i_L(X/Z)$ is just the Fano index of a general fibre of $X\to Z$, see Example~\ref{exa-complete-intersection}.
Recall that the \emph{Fano index} $\fanoind (F)$ of a Fano variety $F$
is the largest natural number $m$ such that $-K_F\equiv m H$
for some Cartier divisor $H$ (cf. 
\cite[Definition V.1.9]{Kollar_rational_curves}). 
The ample divisor $H$ is called a \emph{fundamental divisor} of $F$
(cf. \cite[page 115]{HS19}).


On the other hand, the geometrically irreducible subschemes on degenerations as in Theorem~\ref{kollarAx} 
admit many geometric properties.  For example, the work
\cite{hx07} shows that the geometrically irreducible subscheme 
found by Koll\'ar (in Theorem~\ref{kollarAx}) is actually rationally connected
(see Definition~\ref{rc-var}).
In particular, by applying the techniques from the Minimal Model Program, \cite[Theorem 1.2]{hx07}
generalises the result of Koll\'ar by allowing the base $C$ to be a 
higher dimensional klt $k$-variety 
and the generic fibre $F_{\text{gen}}$ to be rationally connected,
not only smooth and Fano.
Recall that every smooth Fano variety in characteristic zero is rationally connected
(see \cite[Corollary V.2.14]{Kollar_rational_curves} and \cite[Theorem 1]{zhangqi}).


In this paper, we show that the dimension of the geometrically irreducible subscheme
predicted in Theorem~\ref{kollarAx} has a lower bound by the ``Fano index"
of the generic fibre of the fibration.
The following theorem is the main result of this paper.

\begin{theorem}\label{dim-thm} 
	Let $f\colon X\to Z$ be an $L$-polarised g-Fano fibration over 
	a Dedekind scheme $Z$, i.e., a normal Noetherian scheme of dimension one, 
    satisfying the following condition:
    let $z\in Z$ be an arbitrary closed point, then, either
    \begin{itemize}
        \item [\emph{(1)}] the residue field $k(z)$ of $Z$ at $z$ has characteristic zero (e.g., $Z$ is a smooth quasi-projective curve over a field of characteristic zero as in Theorem~\ref{kollarAx}), or
        \item [\emph{(2)}] the residue field $k(z)$ of $Z$ at $z$ has positive characteristic, $k(z)$ contains the algebraic closure of its prime subfield, the function field $K(Z)$ of $Z$ has characteristic zero, and the localisation $\mc{O}_{Z, z}$ is prime regular (see Definition~\ref{primeregular}).
    \end{itemize}
	If 
    \[ 1\le i_L(X/Z) \le \dim(X/Z), \]
	where $\dim (X/Z)$ is the dimension of the general fibre of $f$, then,
	for any closed point $z\in Z$, the fibre $X_z$ contains a geometrically
	irreducible $k(z)$-subscheme $Y_z \subseteq X_z$ such that
    \begin{equation}
        \dim Y_z \ge i_L(X/Z) - 1.  \tag{$\star$}\label{key-condition}
    \end{equation}
	Moreover, if the maximal dimension of the geometrically irreducible $k(z)$-subschemes of $X_z$
	is equal to $i_L(X/Z)-1$, then at least one of these geometrically irreducible 
	$k(z)$-subschemes is $k(z)$-unirational. 
\end{theorem}

\cite[Example 4]{Kollarcom} shows that the only geometrically irreducible subschemes
on degenerations of smooth Fano varieties may be just rational points
(i.e., geometrically irreducible subschemes of dimension zero).  
Our Theorem~\ref{dim-thm} gives a criterion for the existence of \emph{positive dimensional}
geometrically irreducible subschemes on degenerations of Fano varieties.  Moreover, we will show in Example~\ref{counterexample}
that the lower bound on dimension~\eqref{key-condition} in Theorem~\ref{dim-thm}
is sharp.


\mni
\textbf{Construction of this article}.
In Section~\ref{notationsandconventions}, we explain our notation and conventions in this article.
In Section~\ref{preliminaryresults}, we state some preliminary results for the proof of Theorem~\ref{dim-thm}.
Section~\ref{proofofmaintheorem2} gives the proof of Theorem~\ref{dim-thm}.
Examples about Theorem~\ref{dim-thm} are given in Section~\ref{examplesandcounterexamples}.

\mni
\textbf{Acknowledgements:} 
I am very grateful to my PhD advisor, professor Jason Starr,
for introducing this problem and for his constant help during the proof.
I would like to thank my postdoc mentor professor Caucher Birkar for 
many helpful comments on this work.  
I would also like to thank Dingxin Zhang and Zhiyu Tian for useful discussions about the theory of 
(separably) rationally connected varieties.  
I am also very grateful to the anonymous referees for their many corrections and suggestions to improve this paper.


\section{Preliminaries}\label{notationsandconventions}

\subsection{Varieties and schemes}
Every scheme in this article is Noetherian.
By a \emph{variety}, we mean a separated, finite type, 
integral scheme over a field (not necessarily algebraically closed). 
Notice that in our definition a variety does not have to be geometrically integral
(cf. \cite[page 105]{Hart}).
All the morphisms in this article are essentially of finite type.

\subsection{Contractions}
A \emph{contraction} is a projective morphism of schemes $f\colon X\to Y$
such that $f_*\mc{O}_X = \mc{O}_Y$; $f$ is not necessarily birational.
In particular, $f$ has connected fibres.  Moreover, if $X$ is normal, then $Y$ is also normal.

\subsection{Divisors and singularities of pairs}
For (relative) linear equivalence (resp. $\QQ$-linear equivalence, resp. $\RR$-linear equivalence,
resp. numerical equivalence) of divisors, we refer the readers to \cite[\S 2.3]{B-Fano}.
A \emph{$\QQ$-line bundle} $\mc{L}$ on a scheme $X$ is a reflexive sheaf on $X$ such that $\mc{L}^{[m]}$ is an invertible sheaf for some $m\in \NN$, see \cite[\S 1.1]{Kol_singularities_of_MMP}.
For relatively nef (resp. relatively big) divisors,
see \cite[Definition 3.1.1]{bchm}.
For singularities of pairs, such as, klt pairs and lc pairs, see \cite[\S 2.3]{km98}
and \cite[\S 2.8]{B-Fano}.
Notice that, to define singularities of pairs, it suffices to assume that
the ground field of varieties is perfect, see \cite[\S 2.1]{Kol_singularities_of_MMP}.

\subsection{Relatively dense subset and unirationality}
Fix a scheme $S$.  Let $X\to S$ be a separated, finite type $S$-scheme.  
An open subset $U\subset X$ is called 
\emph{$S$-dense} if the fibre $U_s$ is dense in $X_s$ for each $s\in S$.  
We say that $X$ is \emph{$S$-unirational}
if there is an $S$-rational map $\mathbb{P}_S^n\dashrightarrow X$ whose domain 
of definition is $S$-dense
in $\mathbb{P}_S^n$, see \cite[\S 2.5]{BLR}.


\subsection{Grassmannians}

Denote by $\grass (d, N)$ the $\mathbb{Z}$-scheme that represents the functor:
\begin{align*}
\textbf{Sch}&\to\textbf{Set}  \\
S&\mapsto \{\text{rank $(N+1-d)$ sub-bundles of }\mc{O}_S^{N+1}\}.
\end{align*}
For any scheme $S$, denote by $\grass (d, N)_S$ the base change of $\grass (d, N)$ to $S$.
If $S=\spec A$ is affine, we just denote the base change by $\grass (d, N)_A$ for simplicity.
For an $S$-scheme $D$ admitting an $S$-morphism $D \to \grass (d, N)_S$, 
denote by $\mc{X}_D$ the base change
of the universal family $\mc{X}$ over $\grass (d, N)_S$ to $D$.
Note that, for any field $k$, $\grass (d, N)_k$ parameterises codimension $d$ linear subspaces
in $\mathbb{P}_k^N$.


\subsection{Rational-connectivity and separable-rational-connectivity}\label{RC-SRC}
For more general background about (separably) rationally connected and rationally chain connected
varieties, we refer the readers to \cite[\S IV.3]{Kollar_rational_curves}.

\begin{definition}[\protect{\cite[page 3931]{hx07}}]\label{rc-var}
	Let $X$ be a variety over a field $k$.  We say that $X$ is 
	\emph{rationally chain connected} if for an uncountable algebraically closed field $K$
	containing $k$, any two $K$-points of $X$ can be connected by a chain of rational 
	curves defined over $K$.  We say that $X$ is \emph{rationally connected}
	if any two $K$-points of $X$ can be connected by an irreducible 
	rational curve defined over $K$.
\end{definition}


\begin{definition}[cf. \protect{\cite[page 3935]{hx07}}]\label{logfano}
Let $k$ be a perfect field.  A proper, normal variety $X$ over $k$ 
is called \emph{log Fano} (or \emph{log $\mathbb{Q}$-Fano})
if there exists an effective $\mathbb{Q}$-divisor $D$ on $X$ such that $(X, D)$ is klt and 
$-(K_X+D)$ is nef and big.
\end{definition}

When emphasising the divisor $D$, 
we also say that the pair $(X, D)$ is log $\mathbb{Q}$-Fano. 
If $k$ is a field of characteristic zero, 
a log $\mathbb{Q}$-Fano variety is always rationally connected,
see \cite{zhangqi} and \cite{hm07}.  

\begin{definition}\label{weak-Q-Fano}
    A normal projective variety $X$ over a perfect field $k$
    is called \emph{weak $\mathbb{Q}$-Fano} if $X$ has klt singularities and 
    $-K_X$ is nef and big.
    Moreover, we say that $X$ is \emph{$\QQ$-Fano} if $X$ is klt and $-K_X$ is ample.	
\end{definition}

Unfortunately, there is no common convention for the definitions of 
weak $\QQ$-Fano varieties and $\QQ$-Fano varieties in literature.
Some authors drop the singularity requirement,
hence a weak $\QQ$-Fano variety just means a normal projective variety $X$
with $-K_X$ nef and big; however, some authors 
require that the variety $X$ has terminal singularities.
We adopt Definition~\ref{weak-Q-Fano} throughout this paper.
In particular, a weak $\QQ$-Fano variety over a field of characteristic zero
is rationally connected.


The correct notion for rational-connectivity to be studied in
positive characteristic is the so-called 
\emph{separable-rational-connectivity}.
We take the definition of separable-rational-connectivity 
from \cite{deJongStarrofsections} as following.
For smooth varieties over algebraically closed fields, 
this definition agrees with Koll\'ar's definition in \cite[(IV.3.2.3)]{Kollar_rational_curves},
see \cite[Theorem IV.3.7, Theorem IV.3.9.4]{Kollar_rational_curves}.
Separable-rational-connectivity is needed to state the result \cite[Theorem 1.5]{jasonax},
which is one of the ingredients in the proof of Theorem~\ref{dim-thm} 
when the residue fields of the Dedekind scheme $Z$ have positive characteristics.

\begin{definition}[\protect{\cite[page 567]{deJongStarrofsections}} and \protect{\cite[Theorem IV.3.7]{Kollar_rational_curves}}]\label{rationalcnn}
Let $k$ be an algebraically closed field and let $V$ be a normal proper variety over $k$.
We say that $V$ is \emph{separably rationally connected} if
there exists a smooth rational curve $C\to V$ mapping to $V$ which avoids
the singularities of $V$ such that $\mathcal{T}_V|_C$ is ample on $C$.
\end{definition}


\subsection{Prime regular discrete valuation rings}

To address obstructed deformations of separably 
rationally connected varieties 
(in particular, over DVRs of mixed characteristics), we take the definition of
prime regular discrete valuation rings (DVRs) from \cite{Jason:pacfields}.

\begin{definition}[\protect{\cite[page 8]{Jason:pacfields}}]
For DVRs $(\Lambda, \mf{m}_{\Lambda})$ and $(R, \mf{m}_R)$, a local homomorphism $\phi\colon \Lambda\to R$
is \emph{regular} if
\begin{enumerate}[label=(\roman*)]
\item $\phi(\mf{m}_{\Lambda})R$ is equal to $\mf{m}_R$, i.e., $\phi$ is \emph{weakly unramified}
      (note that this implies that $\phi$ is injective),
\item the residue fields extension $\Lambda/\mf{m}_{\Lambda} \to R/\mf{m}_R$ is separable
      (note that this holds automatically if $\Lambda / \mf{m}_{\Lambda}$ is perfect), and 
\item the fraction fields extension is separable
      (note that this holds automatically if the fraction field has characteristic 0).
\end{enumerate}
\end{definition}


\begin{definition}[\protect{\cite[Definition 3.3, page 9]{Jason:pacfields}}]\label{primeregular}
(1). A \emph{prime finite} DVR $(\Lambda, \mf{m}_{\Lambda})$ is a DVR whose residue field $\Lambda/\mf{m}_{\Lambda}$
is a finite extension of the prime subfield, i.e., either the residue field is a finite field if the 
characteristic is positive, or it is a number field if the characteristic is 0.

(2).  A DVR $(R, \mf{m}_R)$ is \emph{prime regular}, or \emph{regular over a DVR whose residue field is 
finite over the prime subfield}, if there exists a prime finite DVR $(\Lambda, \mf{m}_{\Lambda})$
and a local homomorphism $\phi\colon \Lambda\to R$ that is regular.
\end{definition}

\begin{remark}\label{smoothtoprimeregular}
If $(R, \mf{m}_R)$ is a prime regular DVR and $(R, \mf{m}_R)\to (D, \mf{m}_D)$	
is a smooth morphism of DVRs, then $(D, \mf{m}_D)$ 
is also prime regular (see \cite[Remark 3.6]{Jason:pacfields}).
Moreover, every equicharacteristic DVR is prime regular (see \cite[Lemma 3.4]{Jason:pacfields}). 
\end{remark}


\subsection{Ax conjecture in positive characteristics}

Very different approaches have to be employed 
to prove results as Theorem~\ref{kollarAx} in positive characteristics.  
When the field contains an algebraically closed subfield, 
this question is settled in \cite[Theorem 24.3.6 (a)]{fj}.
On the other hand, by using cohomological methods and the technique of decomposition of diagonal,
the results \cite[Theorem 1.1, page 2679]{Esnault_Xu:rational_points_finite_fields} 
and \cite[Corollary 1.2, page 798]{Fakhruddin:2005uo} 
show that a degeneration of a smooth rationally connected variety over a finite field 
must contain a rational point.
Moreover, by a very different approach, the work \cite{jasonax}
shows the existence of a geometrically irreducible subscheme
on degeneration of separably rationally connected varieties
(see Definition~\ref{rationalcnn}).


\begin{theorem}[\protect{\cite[Theorem 1.5]{jasonax}} and \protect{\cite[Remark 3.15]{Jason:pacfields}}]
\label{jason-maintheorem}
Let $R$ be a prime regular DVR (see Definition~\ref{primeregular}) with residue field $k_R$, and let 
$X_R$ be a proper, flat $R$-scheme. 
Denote by $L_R$ the compositum of $k_R$ with the algebraic closure of the prime subfield.
If the geometric generic fibre of $X_R$ over $R$ is separably rationally connected
(in the sense of Definition~\ref{rationalcnn}),
then there exists a closed subscheme $Y_{L_R}$ of $X_R\otimes_R L_R$ which 
is geometrically irreducible over $L_R$, i.e., $Y_{L_R}\otimes_{L_R}\ol{L}_R$
is irreducible.
\end{theorem}

\begin{remark}\label{rem-jason-ax}
    In Theorem~\ref{jason-maintheorem}, when the fraction field of $R$ has characteristic zero,
    it suffices to assume that the geometric generic fibre of $X_R \to \spec R$ is rationally connected
    (see Definition~\ref{rc-var}) by \cite[Theorem 1.1]{GHS} and \cite[\S 5]{jasonax}.
\end{remark}


\subsection{Nice DVRs}
We use the notion of \emph{nice DVR} to refer the discrete valuation rings over which 
Theorem~\ref{kollarAx} and Theorem~\ref{jason-maintheorem} can give the existence
of a geometrically irreducible subscheme on degeneration of 
separably rationally connected varieties.

\begin{definition}\label{nice-dvr}
Let $R$ be a DVR.
We say that $R$ is \emph{nice} if either
\begin{itemize}
\item the residue field $k_R$ of $R$ has characteristic zero, or
\item the fraction field $K(R)$ of $R$ has characteristic zero, $\chara k_R>0$, $R$ is prime regular,
      and $k_R$ contains the algebraic closure of its prime subfield.
\end{itemize}
\end{definition}

Let $R$ be a nice DVR.  If the residue field $k_R$ of $R$
has positive characteristic, then $k_R$ contains the algebraic closure of its prime subfield
by Definition~\ref{nice-dvr}, hence $k_R$ is equal to the compositum of $k_R$
with the algebraic closure of the prime subfield.  In this case,
if $X_R$ is a proper, flat $R$-scheme, whose geometric generic fibre is separably
rationally connected (or, rationally connected as $K(R)$ has characteristic zero, see Remark~\ref{rem-jason-ax}),
then Theorem~\ref{jason-maintheorem} shows the existence of a geometrically irreducible
$k_R$-subscheme of $X_{k_R}$.


\section{Preliminary results}\label{preliminaryresults}

To prove Theorem~\ref{dim-thm}, we first show in the following 
proposition that there exists a subscheme of Grassmannian 
such that the family of linear subspaces parameterised by this subscheme is base-point-free.

\begin{proposition}\label{unirationalparameter}
Fix an arbitrary natural number $N\in \NN$.
Let $R$ be a discrete valuation ring.  Denote by $K$ (resp. $k$) the 
fraction field (resp. the residue field)
of $R$.  Then, for every $d\ge 1$, there exists an integral, normal $R$-scheme $D$
admitting a morphism $D\to \grass (d, N)_R$ over $\spec R$ whose image is a 
locally closed subscheme of $\grass (d, N)_R$ such that
\begin{itemize}
	\item [\emph{(1)}] $D\to \spec R$ is smooth of relative dimension $d$,
	\item [\emph{(2)}] $D\to \spec R$ is $R$-unirational,
	\item [\emph{(3)}] $D\to \spec R$ has geometrically irreducible fibres, and
	\item [\emph{(4)}] the induced morphism $\mc{X}_D\to \mathbb{P}_R^N$ is quasi-finite.
\end{itemize}
Moreover, $D$ can be chosen so that the image of a general point of $D_K$
in $\grass (d, N)_K$ is also a general point of $\grass (d, N)_K$, where 
$D_K$ and $\grass (d, N)_K$ are the fibres of $D$ and $\grass (d, N)_R$ over $\spec K$ respectively.
\end{proposition}

\begin{proof}
We argue by induction on the relative dimension $d$ of $D\to \spec R$.  \\

\emph{Step 1.}
For $d=1$, $\grass (1,N)_K$ is the dual projective 
space $(\mathbb{P}^N_K)^\vee$.  Inside this parameter space, take $D_K$
as any rational normal curve of degree $N$.
There exist homogeneous coordinates on $\mathbb{P}_K^N$, say, 
$[x_0:\ldots :x_N]$, such that $D_K$ is isomorphic to $\mathbb{P}_K^1$ with homogeneous 
coordinates $[t_0: t_1]$ and $\mc{X}_{D_K}$ is the hypersurface in $\mathbb{P}_K^1 \times_{\spec K} \mathbb{P}_K^N$ with 
bihomogeneous defining equation
\[F = t_0^N t_1^0 x_0 + t_0^{N-1} t_1^1 x_1 +\cdots+t_0^{N-m} t_1^m x_m +\cdots+ t_0^0 t_1^N x_N.\]
Since a rational normal curve is nondegenerate, the morphism $\mc{X}_{D_K}\to \mathbb{P}_K^N$ is quasi-finite.
Note that this works over any ground field, even over any ring
since under the homogeneous coordinates above both the rational normal curve $D_K$
and $\mc{X}_{D_K}$ are defined over $\mathbb{Z}$.
Thus, the closure $D$ of $D_K$ (with reduced scheme structure) in $\grass (1, N)_R$ is the required scheme,
where $D_k$ is also a rational normal curve.\\

\emph{Step 2.}
Now, suppose that for some positive integer $d<N$, an integral, normal, $R$-scheme
$D$ of relative dimension $d$ admitting a morphism $D\to\grass (d,N)_R$ 
has been constructed such that all the prescribed properties (1) to (4) of $D\to \spec R$ are satisfied, 
and consider the problem for $d+1$.  

The universal family $\mc{X}_D\subset D \times_{\spec R} \mathbb{P}_R^N$
is a projective space sub-bundle of relative dimension $N-d$.   
Denote by $U:=\spec A$ an affine neighbourhood of the generic point of $D_k$ in $D$.
Then, up to shrinking $U$, we can assume that $\mc{X}_U$ is $A$-isomorphic
to $U\times_{\spec R} \mathbb{P}_R^{N-d}\simeq \mathbb{P}_A^{N-d}$.
Since Step 1 works over any ring, there exists a family $E\simeq\mathbb{P}_A^1$
and a projective space bundle $\mc{X}_E$ over $E$ such that
every geometric fibre of $\mc{X}_E\to E$ over $s\in E$ is 
a codimension $d+1$ linear subspace in $\mathbb{P}^N_{\ol{s}}$ and that
the projection from $\mc{X}_E$ to $\mc{X}_U$ is quasi-finite.  
For convenience, we summarise the objects in the argument in the following diagram.
\[\xymatrix{
\mc{X}_E\ar[d]\ar[r]  &  \grass (1, N-d)_A\times_{\spec A} \mc{X}_U\ar[r]\ar[d]  &  \mc{X}_U\ar[r]\ar[d]  &
\grass(d, N)_R\times_{\spec R} \mathbb{P}_R^N\ar[r]\ar[d]  & \mathbb{P}_R^N\ar[d] \\
E\ar[r]  & \grass(1, N-d)_A\ar[r]  & U\ar[r] & \grass(d, N)_R\ar[r] & \spec R 
}\]
Since the projection from $\mc{X}_U$ to $\mathbb{P}_R^N$ is also quasi-finite, 
it follows that the composite projection from $\mc{X}_E$ to $\mathbb{P}_R^N$ is quasi-finite.
By construction, $E$ is isomorphic to $\mathbb{P}_R^1\times_{\spec R} U$ which is $R$-unirational
as $U$ is also $R$-unirational by induction hypothesis on the $R$-scheme $D$.
Moreover, because both the smooth morphisms $E\to U$ and $U\to \spec R$ 
have geometrically irreducible fibres,
their composite morphism $E\to \spec R$ is smooth and also has geometrically irreducible fibres 
(see \cite[Proposition 4.3.8]{Liuqing}).\\

\emph{Step 3.}
By the universal property of $\grass (d+1, N)_R$, there is an $R$-morphism 
\[ \rho\colon E\to \grass (d+1, N)_R \]
such that $\mc{X}_E$ is the pullback of the universal family over $\grass (d+1, N)_R$ via $\rho$.
Then, $E$ is the desired integral, normal $R$-scheme.
Moreover, by the inductive construction in Step 1 and Step 2,
the $R$-scheme $E$ can be chosen so that the image of a general point of $E_K$ in $\grass (d+1, N)_K$
is also a general point of $\grass (d+1, N)_K$.
\end{proof}


\begin{lemma}\label{qfanofibration}
Let $k$ be a field of characteristic zero, and let $f\colon X\to Y$ be a surjective, 
projective morphism of geometrically integral $k$-varieties
with geometrically irreducible fibres.
Denote by $K(Y)$ the function field of $Y$.
Suppose that there is a dense subset $T\subset Y$ of closed points such that $X_{\ol{t}}$ is
weak $\mathbb{Q}$-Fano
for each $t\in T$.  
Then, the geometric generic fibre $X_{\ol{K(Y)}}$ is also weak $\mathbb{Q}$-Fano.
\end{lemma}

\begin{proof}
We can assume that $k$ is an uncountable algebraically closed field of characteristic zero.
By generic flatness, 
up to shrinking $Y$ to an open dense subset, we can assume further that
\begin{itemize}
	\item $f$ is flat,
    \item $Y$ is nonsingular,
	\item all the geometric fibres of $f$ are normal,
	\item the singular locus of $f$ dominates $Y$, and
	\item $K_{X/Y}$ is $\mathbb{Q}$-Cartier.
\end{itemize}
Take a log resolution of $X$ so that the inverse image of the 
singular locus of $f$ in the resolution is a simple normal crossing divisor, 
then maybe after shrinking $Y$ 
further it is easy to see that we can assume that every 
geometric fibre of $f$ is klt.
By hypothesis, $-K_{X/Y}$
is nef and big on every $X_t$ for $t\in T$.  
By \cite[(1.41)]{km98} and the proof of \cite[(1.4.14)]{robpositivity1}, 
$-K_{X/Y}$ is nef over the generic point of $Y$.
Finally, since $f$ is flat, asymptotic Riemann-Roch (see \cite[(VI.2.15)]{Kollar_rational_curves}) 
shows that $-K_{X/Y}$ is also $f$-big over an open dense subset of $Y$, 
hence $X_{\ol{K(Y)}}$ is weak $\mathbb{Q}$-Fano.
\end{proof}


For our purpose to prove Theorem~\ref{dim-thm}, 
we need a version of Theorem~\ref{kollarAx}
over general DVRs, not only smooth curves, see also \cite[Theorem 1.3]{jasonax}.

\begin{proposition}\label{Kollarhigherdimensional}
	Let $k$ be a field of characteristic zero, and let $R$ be a DVR with residue field $k$.
	Denote by $K$ the fraction field of $R$.  Suppose that $X$ is a proper, flat $R$-scheme
	such that the geometric generic fibre $X_{\ol{K}}$ is weak $\mathbb{Q}$-Fano.
	Then, the closed fibre $X_k$ contains a geometrically irreducible and rationally connected closed $k$-subvariety.
\end{proposition}

\begin{proof}
	By limit argument, we can reduce the problem to the geometric situation 
	(see \cite[Definition 1.7, Lemma 1.8]{jasonax}):
	\begin{itemize}
		\item $S$ is a Dedekind domain of finite type over a field $\kappa$ of characteristic zero and
              $\mf{s}$ is a maximal ideal of $S$,
		\item $P$ is a flat, quasi-projective $S$-scheme whose geometric generic fibre is integral and normal,
		\item $Q$ is an integral, normal Weil divisor of $P$ contained in $P_{\mf{s}}$, and
		\item $X_P$ is a proper, flat $P$-scheme whose geometric generic fibre over $P$ is weak $\mathbb{Q}$-Fano.
	\end{itemize}
	Denote by $X_Q$ the fibre product $X_P\times_P Q$.
	Then, the problem is reduced to find a geometrically irreducible and 
    rationally connected closed $K(Q)$-subvariety of $X_{K(Q)}$.
	Up to shrinking $P$, we can assume that $P$ is smooth, so trivially klt.  
	Recall that weak $\mathbb{Q}$-Fano varieties are rationally connected in characteristic zero.
	Then, we are in the situation to apply \cite[Theorem 1.2]{hx07} so that 
	$X_{K(Q)}$ contains a geometrically irreducible and rationally connected closed $K(Q)$-subvariety.
\end{proof}


\section{Dimensions of subvarieties on degenerations}\label{proofofmaintheorem2}

Now, we give the proof of Theorem~\ref{dim-thm} by the following local version of Theorem~\ref{dim-thm}.

\begin{theorem}\label{maintheorem2}
Let $R$ be a nice DVR.
Denote by $k$ (resp. $K$) the residue field (resp. the fraction field) of $R$.
Suppose that $X$ is a projective and flat $R$-scheme, and let $\mc{L}$ be
a big$/\spec R$ and base-point-free$/\spec R$ invertible sheaf on $X$.  Assume that:
\begin{itemize}
	\item [\emph{(a)}] the geometric generic fibre $X_{\ol{K}}$ is normal, integral, and weak $\mathbb{Q}$-Fano,
    \item [\emph{(b)}] on the generic fibre, the $\QQ$-line bundle $\big(\omega_{X_K}\otimes (\mc{L}|_{X_K})^d\big)^\vee$ is nef and big for some $0\le d< \dim (X/R)$, where $\dim (X/R)$ is the relative dimension of $X\to \spec R$.
\end{itemize}
Then, we have that:
\begin{itemize}
	\item [\emph{(i)}] $X_k$ contains a geometrically irreducible $k$-subscheme with dimension $\ge d$, and
    \item [\emph{(ii)}] if the maximal dimension of geometrically irreducible $k$-subschemes in $X_k$ is equal to $d$, then at least one of these geometrically irreducible $k$-subschemes is $k$-unirational.
\end{itemize}
\end{theorem}

\begin{proof}
If $d=0$, then the conclusions (i) and (ii) are automatic by Proposition~\ref{Kollarhigherdimensional} and
\cite[Theorem 1.5]{jasonax} as $R$ is a nice DVR by 
assumption (see Definition~\ref{nice-dvr}).
So, in the rest of the proof, we assume that $d\ge 1$.\\

\emph{Step 1}.
Since $\mc{L}$ is base-point-free over $\spec R$, it gives an $R$-morphism 
$\phi\colon X\to \mathbb{P}_R^N$ for some integer $N$
so that $\mc{L}$ is the pullback of $\mc{O}_{\mathbb{P}_R^N}(1)$.
Moreover, since $\mc{L}$ is big over $\spec R$, the image of $\phi$ has dimension $\dim X$.
By Bertini's theorem for normality, for a general hyperplane $H\subset \mathbb{P}_K^N$, 
the schematic intersection $X_K\cap H:=\phi_K^{-1}(H)$ is geometrically normal 
(see \cite[Corollary 3.4.9]{normal-bertini}).  
Moreover, since $\dim X_K\ge d+1$, 
by Bertini's theorem, \cite[Corollaire 6.11]{jou}, 
we can assume that $X_K\cap H$
is also geometrically integral.  Finally, since $\chara K=0$ and $\mc{L}$ is base-point-free
over $\spec R$,
by the proof of \cite[Theorem (1.13)]{reid}, the variety $X_K\cap H$ is also klt.  
Therefore, inductively, we can take general members $H_1, \ldots, H_d$ in $\grass (1, N)_K$
such that 
$$Y_K:=X_K\cap H_1\cap \cdots \cap H_d$$ 
is geometrically normal, geometrically integral, and klt.
By taking the hyperplanes generally, we can apply the adjunction formula inductively 
(see \cite[Proposition 4.5]{Kol_singularities_of_MMP}) such that
\begin{align*}
\omega_{Y_K}  &\simeq \omega_{X_K} \otimes \mc{O}(H_1)|_{X_K}\otimes \cdots \otimes \mc{O}(H_d)|_{X_K} \otimes \mc{O}_{Y_K}   \\
    &\simeq \omega_{X_K}\otimes {(\mc{L}|_{X_K})^d}\otimes \mc{O}_{Y_K}.
\end{align*}
By hypothesis, $\big(\omega_{X_K}\otimes (\mc{L}|_{X_K})^d\big)^\vee$ is nef and big, thus
$Y_{\ol{K}}$ is weak $\mathbb{Q}$-Fano.\\

\emph{Step 2}.
Take an integral, normal $R$-scheme $D$ with an $R$-morphism $D\to \grass (d, N)_R$ 
as in Proposition~\ref{unirationalparameter}.
Moreover, by Proposition~\ref{unirationalparameter}, 
we can take $D$ such that the image of a general point of $D_K$
is also general in $\grass (d, N)_K$.  Denote by $\mc{X}_D$ the universal family over $D$; 
we have the following commutative diagram.
\[\xymatrix{
\mc{X}_D\times_{\mathbb{P}_R^N} X\ar[rr]^{\projection_{X}}\ar[d]  &   &  X\ar[d]   \\
\mc{X}_D\ar[r]\ar[rd]  & D\times_{\spec R} \mathbb{P}_R^N\ar[r]\ar[d]  &  \mathbb{P}_R^N\ar[d]   \\
 &  D\ar[r]   &  \spec R 
}\]
Denote by $E$ the localisation of $\mc{O}_D$ at the generic point of $D_k$.
Then, $E$ is a discrete valuation ring whose fraction field is $K(D_K)$ and residue field is $K(D_k)$.
Consider the base change to $\spec E$,
\[\projection_E\colon \mc{X}_D\times_{\mathbb{P}_R^N} X\times_D \spec E \to \spec E.\]
The geometric generic fibre of $\projection_E$ is the geometric generic fibre of 
\[\mc{X}_{D_K}\times_{\mathbb{P}_K^N} X_K \to D_K\]
which is weak $\mathbb{Q}$-Fano by Step 1 and Lemma~\ref{qfanofibration}.
Thus, the geometric generic fibre of $\projection_E$ is rationally connected 
(see \cite{zhangqi} or \cite{hm07}).

Notice that $E$ is also a nice DVR by Remark~\ref{smoothtoprimeregular} 
and Proposition~\ref{unirationalparameter}.
Thus, by Theorem~\ref{jason-maintheorem} (recall that $d\ge 1$), Remark~\ref{rem-jason-ax}, and 
Proposition~\ref{Kollarhigherdimensional},
the closed fibre of $\projection_E$,
which is 
\[\mc{X}_{D_k}\times_{\mathbb{P}_k^N} X_k\times_{D_k} \spec K(D_k) \to \spec K(D_k), \]
contains a geometrically irreducible $K(D_k)$-subvariety $V$.
Since 
\[ \projection_{X_k} \colon \mc{X}_{D_k}\times_{\PP^N_k} X_k
\to X_k \] 
is quasi-finite by Proposition~\ref{unirationalparameter}, 
the closure of the image of the $K(D_k)$-variety $V$ in $X_k$ gives a geometrically irreducible $k$-subvariety 
of $X_k$ with dimension $\ge d$.  This completes the proof of (i).\\

\emph{Step 3}.
For (ii), assume that the geometrically irreducible $K(D_k)$-subvariety 
of $\mc{X}_{D_k}\times_{\mathbb{P}_k^N} X_k\times_{D_k} \spec K(D_k) $
has dimension zero, i.e., it is a closed point $\eta$.  
Since $\spec k(\eta)$ is a geometrically irreducible $K(D_k)$-variety,
it must be a $K(D_k)$-rational point.
Because $D_k$ is unirational, by considering the function fields, we see that
the image $\projection_{X_k}(\spec k(\eta))$ in $X_k$ is also $k$-unirational, hence this
gives a geometrically irreducible $k$-subvariety of $X_k$ which is $k$-unirational.
\end{proof}


\begin{theorem}[=Theorem~\ref{dim-thm}]  
	Let $f\colon X\to Z$ be an $L$-polarised g-Fano fibration over 
	a Dedekind scheme $Z$ whose localisation $\mc{O}_{Z, z}$ is a nice DVR 
    for every closed point $z\in Z$.
	If 
	\[ 1\le i_L(X/Z) \le \dim(X/Z), \]
	where $\dim (X/Z)$ is the dimension of the general fibre of $f$, then,
	for any closed point $z\in Z$, the fibre $X_z$ contains a geometrically
	irreducible $k(z)$-subscheme $Y_z \subseteq X_z$,
	where $k(z)$ is the residue field of $Z$ at $z$, such that
	\[ \dim Y_z \ge i_L(X/Z) - 1. \]
	Moreover, if the maximal dimension of the geometrically irreducible $k(z)$-subschemes of $X_z$
	is equal to $i_L(X/Z)-1$, then at least one of these geometrically irreducible 
	$k(z)$-subschemes is $k(z)$-unirational. 
\end{theorem}

\begin{proof}
    For the fixed $z\in Z$, set $R := \mc{O}_{Z, z}$, which is a nice DVR by assumption.
	Set
	\[ d:= i_L(X/Z) - 1, \]
	then $0\le d < \dim (X/Z)$.  Let $K$ be the fraction field of $R$
	(i.e., the function field $K(Z)$ of $Z$).  Then, we have that
	\[ (\omega_{X_K}\otimes \mc{O}_{X_K}(L|_{X_K})^d)^{\vee} \equiv \mc{O}_{X_K}(L|_{X_K}), \]
	which is nef and big since $-K_{X_K} \equiv i_L(X/Z) L|_{X_K}$ over $K(Z)$.
	Recall that the geometric generic fibre of $f$ is weak $\QQ$-Fano (see Definition~\ref{polarised-g-fano}),
	then the result follows immediately from Theorem~\ref{maintheorem2}.
\end{proof}


\section{Examples}\label{examplesandcounterexamples}

The following example shows that, in certain situations, such as, the construction in the proof of Corollary~\ref{Ax-chara-zero},
the $L$-index of an $L$-polarised g-Fano fibration $X\to Z$ is equal to
the Fano index of the general fibre of $X\to Z$.

\begin{example}\label{exa-complete-intersection}
    Let $n$ be a natural number.  Let $Z$ be a smooth quasi-projective curve over a 
    field $k$ of characteristic zero.
    Let $H_1, \dots, H_r$ be ample$/Z$ divisors on $Z\times_{\spec k} \PP^n_k$
    with relative degrees $d_1, \dots, d_r$ (with respect to the projection $Z\times_{\spec k} \PP^n_k \to Z$).
    Moreover, assume that every $H_i$ is smooth over the generic point of $Z$ for every $1\le i\le r$.

    Denote by $X := H_1\cap \cdots \cap H_r$. 
    Let $L$ be the pullback of a general hyperplane of $\PP_k^n$ to $X$
    via $X \to Z\times_{\spec k} \PP^n_k \to \PP^n_k$, 
    then $L$ is a very ample$/Z$ Cartier divisor on $X$.
    Assume that a general fibre $F$ of $X\to Z$ is a smooth complete intersection 
    of dimension $\ge 3$
    defined by the smooth hypersurfaces $H_1|_F, \dots, H_r|_F$. 
    Assume that $\sum_{i=1}^r d_i < n+1$, then $F$ is
    a smooth Fano variety.
    Moreover, as $\dim F \ge 3$, the Picard number $\rho(F)$
    of $F$ is equal to one by Lefschetz theorem
    (see \cite[Example 3.1.25]{robpositivity1}),
    hence $i_L(X/Z)$ is a well-defined non-negative integer.  
    It is easy to see that
    \[ i_L(X/Z) = \fanoind (F) = n+1 - \sum_{i=1}^r d_i \]
    by \cite[(V.1.9.1)]{Kollar_rational_curves}, where $\fanoind (F)$ is the Fano index of $F$.
    Notice that $L|_F$ is a fundamental divisor of $F$.

    For any fixed closed point $z\in Z$, the $k(z)$-variety $X_z$ is the degeneration of
    smooth Fano varieties that are complete intersections.  
    Then, Theorem~\ref{dim-thm} predicts that $X_z$
    contains a geometrically irreducible $k(z)$-subvariety of dimension at least $\fanoind (F) - 1$. 
\end{example}


Now, we show that the lower bound on dimension~\eqref{key-condition} 
in Theorem~\ref{dim-thm} is sharp by giving an example of
an $L$-polarised g-Fano fibration
whose $L$-index is equal to one and the only 
geometrically irreducible closed subscheme on the degeneration 
is a rational point.
In particular, when the $L$-index is one, it is possible that
there is no \emph{positive dimensional} geometrically irreducible subscheme on the special fibre.

\begin{example}\label{counterexample}
Let $k=\mathbb{Q}$, $C=\mathbb{P}^1_k$, and in $\mathbb{P}^1_k\times_{\spec k} \mathbb{P}^4_k$ 
consider the family of 3-folds
\[ X := \big(s \big( (x^2+y^2)^2+(z^2+w^2)^2\big)+tu^4=0 \big)\subset \mathbb{P}^1_k\times_{\spec k} \mathbb{P}^4_k,\]
with projection $g\colon X\to C$,
where $[s:t]$ and $[x: y: z: w: u]$ are the homogeneous coordinates of $\mathbb{P}_k^1$ and $\mathbb{P}_k^4$ respectively.  
Let $L$ be the pullback of a general hyperplane of  $\PP_k^4$ via the projection $X \to \PP_k^4$.
Then, $g\colon X\to C$ is an $L$-polarised g-Fano fibration 
whose general fibres are smooth Fano 3-folds.
By Example~\ref{exa-complete-intersection}, we have
\[ i_L(X/C) = 1. \]
Set $c:= [1: 0]$.
We claim that the only geometrically irreducible $\mathbb{Q}$-subvariety of $X_c$
is the $\mathbb{Q}$-rational point $P:=[1:0]\times[0:0:0:0:1]\in \mathbb{P}^1_k\times_{\spec k} \mathbb{P}^4_k$.

First, the 3-fold $X_c$ is irreducible but not geometrically irreducible since after base change to $\mathbb{Q}[i]$
there are two irreducible components (with $i = \sqrt{-1}$),
\begin{align*}
X_c^+ &:=(x^2+y^2+i(z^2+w^2)=0)\subset \mathbb{P}^4_{\mathbb{Q}[i]},  \\
X_c^-  &:=(x^2+y^2-i(z^2+w^2)=0)\subset \mathbb{P}^4_{\mathbb{Q}[i]},
\end{align*}
which are conjugate under the group action of $\galois (\mathbb{Q}[i]/\mathbb{Q})$.
Moreover, both $X_c^+$ and $X_c^-$ map surjectively onto $X_c$ in the following Cartesian diagram.
\[\xymatrix{
X_c^+\cup X_c^-\ar[r]\ar[d]  &  X_c\ar[d]  \\
\spec \mathbb{Q}[i]\ar[r]  &  \spec \mathbb{Q}
}\]
Suppose that $X_2$ is a geometrically irreducible $\mathbb{Q}$-subvariety of $X_c$.
Then, $X_2\otimes_{\mathbb{Q}}\mathbb{Q}[i]$ is $\galois (\mathbb{Q}[i]/\mathbb{Q})$-invariant
and should be contained in 
\[X_c^+\cap X_c^-=( x^2+y^2=z^2+w^2=0)\subset \mathbb{P}^4_{\mathbb{Q}[i]}.\]
Thus, $X_2\otimes_{\mathbb{Q}}\mathbb{Q}[i]$ is one of the irreducible components 
of $X_c^+\cap X_c^-$
which are 2-dimensional linear subspaces of $\mathbb{P}^4_{\mathbb{Q}[i]}$.
However, none of these linear subspaces is $\galois (\mathbb{Q}[i]/\mathbb{Q})$-invariant, hence $X_c$ has no geometrically irreducible $\mathbb{Q}$-subvariety of dimension two.
The same argument shows that $X_c$ does not admit one-dimensional geometrically irreducible 
$\mathbb{Q}$-subvariety neither.
Finally, by dimension reason, the only zero-dimensional geometrically irreducible $\mathbb{Q}$-subvariety
is a $\mathbb{Q}$-rational point, and the only possibility is $P=[1:0]\times[0:0:0:0:1]$.
\end{example}

\medskip

\printbibliography

@Article{GHS,
 Author = {Graber, Tom and Harris, Joe and Starr, Jason},
 Title = {Families of rationally connected varieties},
 FJournal = {Journal of the American Mathematical Society},
 Journal = {J. Am. Math. Soc.},
 ISSN = {0894-0347},
 Volume = {16},
 Number = {1},
 Pages = {57--67},
 Year = {2003},
 Language = {English},
 DOI = {10.1090/S0894-0347-02-00402-2},
 zbMATH = {1832408},
 Zbl = {1092.14063}
}

@book {normal-bertini,
    AUTHOR = {Flenner, H. and O'Carroll, L. and Vogel, W.},
     TITLE = {Joins and intersections},
    SERIES = {Springer Monographs in Mathematics},
 PUBLISHER = {Springer-Verlag, Berlin},
      YEAR = {1999},
     PAGES = {vi+307},
      ISBN = {3-540-66319-3},
   MRCLASS = {14C17 (13H15 14-02)},
  MRNUMBER = {1724388},
MRREVIEWER = {L\^{e} Tu\^{a}n Hoa},
       DOI = {10.1007/978-3-662-03817-8},
       URL = {https://doi.org/10.1007/978-3-662-03817-8},
}

@Article{DJL83,
 Author = {Denef, Jan and Jarden, Moshe and Lewis, D. J.},
 Title = {On {Ax}-fields which are {{\(C_ i\)}}},
 FJournal = {The Quarterly Journal of Mathematics. Oxford Second Series},
 Journal = {Q. J. Math., Oxf. II. Ser.},
 ISSN = {0033-5606},
 Volume = {34},
 Pages = {21--36},
 Year = {1983},
 Language = {English},
 DOI = {10.1093/qmath/34.1.21},
 zbMATH = {3821890},
 Zbl = {0519.12015}
}

@Article{HS19,
 Author = {H{\"o}ring, Andreas and {\'S}miech, Robert},
 Title = {Anticanonical system of {Fano} fivefolds},
 FJournal = {Mathematische Nachrichten},
 Journal = {Math. Nachr.},
 ISSN = {0025-584X},
 Volume = {293},
 Number = {1},
 Pages = {115--119},
 Year = {2020},
 Language = {English},
 DOI = {10.1002/mana.201900311},
 zbMATH = {7197938},
 Zbl = {1486.14058}
}

@incollection {reid,
    AUTHOR = {Reid, Miles},
     TITLE = {Canonical {$3$}-folds},
 BOOKTITLE = {Journ\'{e}es de {G}\'{e}ometrie {A}lg\'{e}brique d'{A}ngers,
              {J}uillet 1979/{A}lgebraic {G}eometry, {A}ngers, 1979},
     PAGES = {273--310},
 PUBLISHER = {Sijthoff \& Noordhoff, Alphen aan den Rijn---Germantown, Md.},
      YEAR = {1980},
      ISBN = {90-286-0500-2},
   MRCLASS = {14J30 (14B07 14J17)},
  MRNUMBER = {605348},
MRREVIEWER = {Jayant\ M.\ Shah},
}

@article {hm07,
    AUTHOR = {Hacon, Christopher D. and Mckernan, James},
     TITLE = {On {S}hokurov's rational connectedness conjecture},
   JOURNAL = {Duke Math. J.},
  FJOURNAL = {Duke Mathematical Journal},
    VOLUME = {138},
      YEAR = {2007},
    NUMBER = {1},
     PAGES = {119--136},
      ISSN = {0012-7094,1547-7398},
   MRCLASS = {14E30 (14E05 14J45)},
  MRNUMBER = {2309156},
MRREVIEWER = {Mihnea\ Popa},
       DOI = {10.1215/S0012-7094-07-13813-4},
       URL = {https://doi.org/10.1215/S0012-7094-07-13813-4},
}

@book {jou,
    AUTHOR = {Jouanolou, Jean-Pierre},
     TITLE = {Th\'{e}or\`emes de {B}ertini et applications},
    SERIES = {Progress in Mathematics},
    VOLUME = {42},
 PUBLISHER = {Birkh\"{a}user Boston, Inc., Boston, MA},
      YEAR = {1983},
     PAGES = {ii+127},
      ISBN = {0-8176-3164-X},
   MRCLASS = {13C10 (14-02 14C99)},
  MRNUMBER = {725671},
MRREVIEWER = {Allen\ B.\ Altman},
}

@article {Fakhruddin:2005uo,
    AUTHOR = {Fakhruddin, N. and Rajan, C. S.},
     TITLE = {Congruences for rational points on varieties over finite
              fields},
   JOURNAL = {Math. Ann.},
  FJOURNAL = {Mathematische Annalen},
    VOLUME = {333},
      YEAR = {2005},
    NUMBER = {4},
     PAGES = {797--809},
      ISSN = {0025-5831,1432-1807},
   MRCLASS = {14G15 (11G25 14C15 14G05)},
  MRNUMBER = {2195144},
MRREVIEWER = {Chandan\ Singh\ Dalawat},
       DOI = {10.1007/s00208-005-0697-4},
       URL = {https://doi.org/10.1007/s00208-005-0697-4},
}

@article {Esnault_Xu:rational_points_finite_fields,
    AUTHOR = {Esnault, H\'{e}l\`ene and Xu, Chenyang},
     TITLE = {Congruence for rational points over finite fields and coniveau
              over local fields},
   JOURNAL = {Trans. Amer. Math. Soc.},
  FJOURNAL = {Transactions of the American Mathematical Society},
    VOLUME = {361},
      YEAR = {2009},
    NUMBER = {5},
     PAGES = {2679--2688},
      ISSN = {0002-9947,1088-6850},
   MRCLASS = {14G20 (14G05 14G15)},
  MRNUMBER = {2471935},
MRREVIEWER = {Carlos\ Casta\~{n}o-Bernard},
       DOI = {10.1090/S0002-9947-08-04629-1},
       URL = {https://doi.org/10.1090/S0002-9947-08-04629-1},
}

@article {zhangqi,
    AUTHOR = {Zhang, Qi},
     TITLE = {Rational connectedness of log {${\bf Q}$}-{F}ano varieties},
   JOURNAL = {J. Reine Angew. Math.},
  FJOURNAL = {Journal f\"{u}r die Reine und Angewandte Mathematik. [Crelle's
              Journal]},
    VOLUME = {590},
      YEAR = {2006},
     PAGES = {131--142},
      ISSN = {0075-4102,1435-5345},
   MRCLASS = {14E30 (14J45)},
  MRNUMBER = {2208131},
MRREVIEWER = {Stefan\ Kebekus},
       DOI = {10.1515/CRELLE.2006.006},
       URL = {https://doi.org/10.1515/CRELLE.2006.006},
}

@article {deJongStarrofsections,
    AUTHOR = {de Jong, A. J. and Starr, J.},
     TITLE = {Every rationally connected variety over the function field of
              a curve has a rational point},
   JOURNAL = {Amer. J. Math.},
  FJOURNAL = {American Journal of Mathematics},
    VOLUME = {125},
      YEAR = {2003},
    NUMBER = {3},
     PAGES = {567--580},
      ISSN = {0002-9327,1080-6377},
   MRCLASS = {14D06 (14G27)},
  MRNUMBER = {1981034},
MRREVIEWER = {M.\ Kh.\ Gizatullin},
       URL = {http://muse.jhu.edu/journals/american_journal_of_mathematics/v125/125.3jong.pdf},
}

@misc{Jason:pacfields,
      title={Rationally simply connected varieties and pseudo algebraically closed fields}, 
      author={Jason Michael Starr},
      year={2017},
      eprint={1704.02932},
      archivePrefix={arXiv},
      primaryClass={math.AG}
}

@incollection {jasonax,
    AUTHOR = {Starr, Jason},
     TITLE = {Degenerations of rationally connected varieties and {PAC}
              fields},
 BOOKTITLE = {A celebration of algebraic geometry},
    SERIES = {Clay Math. Proc.},
    VOLUME = {18},
     PAGES = {577--589},
 PUBLISHER = {Amer. Math. Soc., Providence, RI},
      YEAR = {2013},
      ISBN = {978-0-8218-8983-1},
   MRCLASS = {14G05 (12E30 14G27 14M22)},
  MRNUMBER = {3114958},
MRREVIEWER = {Yong\ Hu},
}

@article {hx07,
    AUTHOR = {Hogadi, Amit and Xu, Chenyang},
     TITLE = {Degenerations of rationally connected varieties},
   JOURNAL = {Trans. Amer. Math. Soc.},
  FJOURNAL = {Transactions of the American Mathematical Society},
    VOLUME = {361},
      YEAR = {2009},
    NUMBER = {7},
     PAGES = {3931--3949},
      ISSN = {0002-9947,1088-6850},
   MRCLASS = {14M22 (14D06 14F45)},
  MRNUMBER = {2491906},
MRREVIEWER = {Alexandr\ V.\ Pukhlikov},
       DOI = {10.1090/S0002-9947-09-04715-1},
       URL = {https://doi.org/10.1090/S0002-9947-09-04715-1},
}

@article {Kollarcom,
    AUTHOR = {Koll\'{a}r, J\'{a}nos},
     TITLE = {A conjecture of {A}x and degenerations of {F}ano varieties},
   JOURNAL = {Israel J. Math.},
  FJOURNAL = {Israel Journal of Mathematics},
    VOLUME = {162},
      YEAR = {2007},
     PAGES = {235--251},
      ISSN = {0021-2172,1565-8511},
   MRCLASS = {14D06 (12E30 14G05 14J45)},
  MRNUMBER = {2365862},
MRREVIEWER = {Tam\'{a}s\ Szamuely},
       DOI = {10.1007/s11856-007-0097-4},
       URL = {https://doi.org/10.1007/s11856-007-0097-4},
}

@article {ax68,
    AUTHOR = {Ax, James},
     TITLE = {The elementary theory of finite fields},
   JOURNAL = {Ann. of Math. (2)},
  FJOURNAL = {Annals of Mathematics. Second Series},
    VOLUME = {88},
      YEAR = {1968},
     PAGES = {239--271},
      ISSN = {0003-486X},
   MRCLASS = {10.80},
  MRNUMBER = {229613},
MRREVIEWER = {D.\ J.\ Lewis},
       DOI = {10.2307/1970573},
       URL = {https://doi.org/10.2307/1970573},
}

@Book{fj,
 Author = {Fried, Michael D. and Jarden, Moshe},
 Title = {Field arithmetic},
 Edition = {4th corrected edition},
 FSeries = {Ergebnisse der Mathematik und ihrer Grenzgebiete. 3. Folge},
 Series = {Ergeb. Math. Grenzgeb., 3. Folge},
 ISSN = {0071-1136},
 Volume = {11},
 Year = {2023},
 Publisher = {Cham: Springer},
 Language = {English},
 DOI = {10.1007/978-3-031-28020-7},
 zbMATH = {7720284}
}

@Article{B-Fano,
 Author = {Birkar, Caucher},
 Title = {Anti-pluricanonical systems on {Fano} varieties},
 FJournal = {Annals of Mathematics. Second Series},
 Journal = {Ann. Math. (2)},
 ISSN = {0003-486X},
 Volume = {190},
 Number = {2},
 Pages = {345--463},
 Year = {2019},
 Language = {English},
 DOI = {10.4007/annals.2019.190.2.1},
 zbMATH = {7107180},
 Zbl = {1470.14078}
}

@book {BLR,
    AUTHOR = {Bosch, Siegfried and L\"{u}tkebohmert, Werner and Raynaud,
              Michel},
     TITLE = {N\'{e}ron models},
    SERIES = {Ergebnisse der Mathematik und ihrer Grenzgebiete (3) [Results
              in Mathematics and Related Areas (3)]},
    VOLUME = {21},
 PUBLISHER = {Springer-Verlag, Berlin},
      YEAR = {1990},
     PAGES = {x+325},
      ISBN = {3-540-50587-3},
   MRCLASS = {14K15 (11G10 14L15)},
  MRNUMBER = {1045822},
MRREVIEWER = {James\ Milne},
       DOI = {10.1007/978-3-642-51438-8},
       URL = {https://doi.org/10.1007/978-3-642-51438-8},
}

@book{Kol_singularities_of_MMP,
    AUTHOR = {Koll\'{a}r, J\'{a}nos},
     TITLE = {Singularities of the minimal model program},
    SERIES = {Cambridge Tracts in Mathematics},
    VOLUME = {200},
      NOTE = {With a collaboration of S\'{a}ndor Kov\'{a}cs},
 PUBLISHER = {Cambridge University Press, Cambridge},
      YEAR = {2013},
     PAGES = {x+370},
      ISBN = {978-1-107-03534-8},
   MRCLASS = {14E30 (14B05)},
  MRNUMBER = {3057950},
MRREVIEWER = {Tommaso\ De Fernex},
       DOI = {10.1017/CBO9781139547895},
       URL = {https://doi.org/10.1017/CBO9781139547895},
}

@Book{km98,
 Author = {Koll{\'a}r, J{\'a}nos and Mori, Shigefumi},
 Title = {Birational geometry of algebraic varieties. {With} the collaboration of {C}. {H}. {Clemens} and {A}. {Corti}},
 Edition = {Paperback reprint of the hardback edition 1998},
 FSeries = {Cambridge Tracts in Mathematics},
 Series = {Camb. Tracts Math.},
 ISSN = {0950-6284},
 Volume = {134},
 ISBN = {978-0-521-06022-6},
 Year = {2008},
 Publisher = {Cambridge: Cambridge University Press},
 Language = {English},
 zbMATH = {5273473},
 Zbl = {1143.14014}
}

@Article{bchm,
 Author = {Birkar, Caucher and Cascini, Paolo and Hacon, Christopher D. and McKernan, James},
 Title = {Existence of minimal models for varieties of log general type},
 FJournal = {Journal of the American Mathematical Society},
 Journal = {J. Am. Math. Soc.},
 ISSN = {0894-0347},
 Volume = {23},
 Number = {2},
 Pages = {405--468},
 Year = {2010},
 Language = {English},
 DOI = {10.1090/S0894-0347-09-00649-3},
 zbMATH = {5775673},
 Zbl = {1210.14019}
}

@book {robpositivity1,
    AUTHOR = {Lazarsfeld, Robert},
     TITLE = {Positivity in algebraic geometry. {I}},
    SERIES = {Ergebnisse der Mathematik und ihrer Grenzgebiete. 3. Folge. A
              Series of Modern Surveys in Mathematics [Results in
              Mathematics and Related Areas. 3rd Series. A Series of Modern
              Surveys in Mathematics]},
    VOLUME = {48},
      NOTE = {Classical setting: line bundles and linear series},
 PUBLISHER = {Springer-Verlag, Berlin},
      YEAR = {2004},
     PAGES = {xviii+387},
      ISBN = {3-540-22533-1},
   MRCLASS = {14-02 (14C20)},
  MRNUMBER = {2095471},
MRREVIEWER = {Mihnea\ Popa},
       DOI = {10.1007/978-3-642-18808-4},
       URL = {https://doi.org/10.1007/978-3-642-18808-4},
}

@Book{Liuqing,
 Author = {Liu, Qing},
 Title = {Algebraic geometry and arithmetic curves. {Transl}. by {Reinie} {Ern{\'e}}},
 FSeries = {Oxford Graduate Texts in Mathematics},
 Series = {Oxf. Grad. Texts Math.},
 Volume = {6},
 ISBN = {0-19-920249-4},
 Year = {2006},
 Publisher = {Oxford: Oxford University Press},
 Language = {English},
 zbMATH = {5048200},
 Zbl = {1103.14001}
}

@book {Kollar_rational_curves,
    AUTHOR = {Koll\'{a}r, J\'{a}nos},
     TITLE = {Rational curves on algebraic varieties},
    SERIES = {Ergebnisse der Mathematik und ihrer Grenzgebiete. 3. Folge. A
              Series of Modern Surveys in Mathematics [Results in
              Mathematics and Related Areas. 3rd Series. A Series of Modern
              Surveys in Mathematics]},
    VOLUME = {32},
 PUBLISHER = {Springer-Verlag, Berlin},
      YEAR = {1996},
     PAGES = {viii+320},
      ISBN = {3-540-60168-6},
   MRCLASS = {14-02 (14C05 14E05 14F17 14J45)},
  MRNUMBER = {1440180},
MRREVIEWER = {Yuri\ G.\ Prokhorov},
       DOI = {10.1007/978-3-662-03276-3},
       URL = {https://doi.org/10.1007/978-3-662-03276-3},
}

@Book{Hart,
 Author = {Hartshorne, Robin},
 Title = {Algebraic geometry. {Corr}. 3rd printing},
 FSeries = {Graduate Texts in Mathematics},
 Series = {Grad. Texts Math.},
 ISSN = {0072-5285},
 Volume = {52},
 Year = {1983},
 Publisher = {Springer, Cham},
 Language = {English},
 zbMATH = {3842033},
 Zbl = {0531.14001}
}








 \vspace{1em}
 
\noindent\small{Santai Qu} 

\noindent\small{\textsc{Institute of Geometry and Physics, University of Science and Technology of China, No. 96 Jinzhai Road, Hefei, Anhui Province, 230026, China} }

\noindent\small{Email: \texttt{santaiqu@ustc.edu.cn}}

 \end{document}